\documentclass[11pt]
{article}
\usepackage{color,float,soul,graphicx}%
\usepackage{multirow,enumerate,enumitem}
\usepackage{url,cite,fullpage,fancyhdr}
\usepackage{amsmath,amssymb,amsfonts,amsthm,mathrsfs,amscd}%
\usepackage[doc]{optional}
\usepackage{mathtools,esvect}
\usepackage{cite}
\usepackage{color}
\usepackage{float}
\usepackage{subcaption}
\usepackage{soul}
\usepackage{xcolor,textcomp,manyfoot}%
\usepackage{booktabs}%
\usepackage{algorithm,algorithmicx}%
\usepackage{algpseudocode,listings}%
\usepackage{fancyhdr}
\newcommand{\fenv}[1]%
{\ensuremath{\,\overrightarrow{\operatorname{env}}_{#1}}}
\newcommand{\benv}[1]%
{\ensuremath{\,\overleftarrow{\operatorname{env}}_{#1}}}

\renewcommand\theenumi{(\roman{enumi})}

\theoremstyle{thmstyleone}%
\theoremstyle{thmstyletwo}%
\theoremstyle{thmstylethree}%
\newtheorem{theorem}{Theorem}
\newtheorem{lemma}[theorem]{Lemma}

\newtheorem{proposition}[theorem]{Proposition}%
\newtheorem{definition}{Definition}%
\newtheorem{remark}{Remark}%
\newtheorem{alg}{Algorithm}[section]

\usepackage{cite}
\theoremstyle{definition}

\newcommand{\argmin}{{\rm arg}\!\min}

\newcommand{\nexto}{\kern -0.54em}

\newcommand{\dZ}{{\cal Z \kern -0.7em Z}}
\newcommand{\dC}{{\rm\hbox{C \kern-0.8em\raise0.2ex\hbox{\vrule height5.4pt width0.7pt}}}}
\newcommand{\dQ}{{\rm\hbox{Q \kern-0.85em\raise0.25ex\hbox{\vrule height5.4pt width0.7pt}}}}

\usepackage{epsfig}
\usepackage{latexsym}

\newcommand{\HH}{\mathcal{H}}
\newcommand{\RR}{\mathbb{R}}

\date{}
\begin{document}
\title{Iterative Schemes for Uniformly Nonconvex Equilibrium Problems}
\author{Oday Hazaimah
\footnote{E-mail: {\tt odayh982@yahoo.com}. https://orcid.org/0009-0000-8984-2500. 
St. Louis, MO, USA}}
%

\maketitle\thispagestyle{fancy}
\maketitle
\begin{abstract}\noindent Uniformly regular equilibrium problems are natural generalizations of abstract equilibrium problems and they are defined over the uniformly prox-regular nonconvex sets. Some new efficient implicit methods for solving uniformly regular equilibrium problems are analyzed by the auxiliary principle and inertial proximal methods. The convergence analysis of the new proposed methods is considered under some mild conditions. Gap functions are constructed to suggest some descent-type scheme for uniformly regular equilibrium problems. Our results can be viewed as significant refinements and improvements of the previously known results and they continue to hold for equilibrium problems, variational inequalities and complementarity problems as well.
\medskip 

\medskip 

\end{abstract}

\maketitle

\section{Introduction}
Equilibrium is a concept widely employed and analysed for phenomena arising in operations research, transportation networks, financial situations and economics, to name a few, \cite{Giannessi et al, Mastroeminimax, Mastroeni,Nagurney}. Equilibrium problems by their nature tend to be complex and, often large-scale. Hence, the study of equilibrium problems has been challenging both from modeling as well as computational perspectives and they can be formulated as systems of nonlinear equations, optimization problems, complementarity problems, fixed point problems, or, more recently, variational inequalities for static and dynamic situations. Variational inequalities have emerged as a fruitful branch of applied mathematics with a wide range of real-world applications in industry, finance and engineering. 
Variational inequality theory can be considered as a natural generalization of the classical variational principles. To study and analyse the existence, stability and continuity dependence of the solution of variational inequalities, several numerical techniques were proposed including; (i) projection methods and its variants including Wiener-Hopf equations, (ii) resolvent dynamical systems are fulfilled by establishing the equivalence between variational inequalities and fixed-point problems, see for instance, \cite{Alvarez,Noor,NoorOettli,Giannessi,Glowinski,ZhuMarcotte}. It is well known that the convergence of the projection methods requires that the operator must be strongly monotone and Lipschitz continuous. Unfortunately these strict conditions rule out many applications of this type of methods. The extragradient-type methods overcome this difficulty by performing an additional forward step \cite{Korpelevich} (i.e., double projections). It is well-known that projection and resolvent type methods cannot be extended for mixed quasi variational inequalities. To overcome this drawback, Glowinski et al. \cite{Glowinski} used the auxiliary principle technique for variational inequalities, whereas Noor \cite{Noor,Noorequilibrium} modified the same technique to suggest some inertial proximal methods for equilibrium problems. In this paper, we modify the auxiliary principle technique to analyse novel iterative methods for wide class of equilibrium problems known as uniformly regular equilibrium problems. The convergence criteria of these methods is analyzed under some mild conditions. As a consequence of this approach, we construct the gap (merit) function for uniformly regular equilibrium problems, which can be used to develop descent-type methods for solving nonconvex equilibrium problems.  
Many numerical algorithms in convex optimization used the projection operator over a set. This is the case, for example, of the proximal point algorithm, the gradient projection algorithm (to name just a few). The class of nonempty closed convex sets of a Hilbert space provides a good example too. In order to go beyond the convexity, the class of uniformly prox-regular sets appear as a fundamental concept and it is much larger than that of convex sets, but shares with it many significant properties that are important to the applications in optimization, differential inclusions and control theory \cite{Adly,Clarke et al}. It is clear that the uniformly regular class includes several convex and nonconvex classes such as; $p$-convex sets, weakly convex sets, compact $C^2$ manifolds, proximally smooth sets, the images under a $C^{1,1}$ diffeomorphism of convex sets, see for more details \cite{Clarke et al, Noorequilibrium, PoliRockaThiba} and the references therein. Proximal regularity of a set at a point is a variational condition related to normal vectors which is common to many types of sets. A set $K\in\RR^n$ is prox-regular at a point $x\in K$ for a vector $v$ if and only if its indicator function $\delta_K$ is prox-regular function at $x$ for $v$. In particular a set is prox-regular if and only if the projection mapping is locally single-valued. Similarly, convex and convex-like functions can be captured in a larger class than that of convex functions, the class of nonconvex functions. A function is prox-regular at $\Bar{x}$ for the subgradient $\Bar{v}$ if and only if its epigraph is prox-regular at $(\Bar{x}, f(\Bar{x}))$ for the normal vector.
The class of prox-regular functions covers all lower semi-continuous, proper, convex functions, lower-$C^2$ functions, strongly amenable functions (i.e. convexly composite functions), hence a large core of functions of interest in variational analysis and optimization. These functions, despite being in general nonconvex, possess many of the properties that one would expect only to find in convex or near convex (lower-$C^2$) functions, for instance the Moreau-envelopes are $C^1$. There is an ongoing need, has arisen in the literature, for establishing equivalent formulations of equilibrium problems. One popular strategy is by constructing gap functions and this is called the regularization which is obtained by adding a strictly convex term $G(x,y)$ to the associated operator $F.$ However, practically, it is hard to satisfy the strict convexity assumption. To overcome this difficulty, an auxiliary equilibrium problem is introduced which coincides with the original equilibrium problem.
Merit (or gap) functions known as crucial tools to build bridges between variational inequalities, broadly-classified, and optimization-related problems.  
Another corresponding formulation of equilibrium problems through minimax (saddle) representation \cite{Mastroeni}. Fukushima \cite{Fukushima} defined a continuously differentiable gap function for variational inequalities by introducing an equivalent auxiliary equilibrium problem to the original one. 

The aim of this paper is to investigate and analyse a new algorithmic-approach designed to approximate the solution to a uniformly prox-regular equilibrium problem \eqref{unifreg}. Furthermore, we construct the gap (merit) function associated to uniformly regular equilibrium by using the auxiliary principle technique. This equivalent optimization-equilibrium reformulation can be used to develop descent-type algorithm for solving uniformly regular equilibrium problems under suitable conditions on the function $F$ and can further attack several problems by changing conditions on the parameters. It can be expected that the techniques described in this work will be useful for more elaborate dynamical models, and that the connection between such dynamical models and the solutions to the variational inequalities (besides a variety of variational classes) will provide a deeper understanding of general/mixed equilibrium problems beyond convexity.
\medskip

This paper is organized as follows: Next section exhibits basic facts and important definitions that will be used throughout the entire paper; Section $3$ proposes some algorithms for uniformly regular equilibrium problems and analyzes their convergence; Section $4$ deals with gap functions and suggests descent-type algorithm viewed in the optimization framework.

\section{Preliminaries} 
Some basic concepts and essential definitions are introduced in this section from the fields of nonsmooth analysis and variational analysis, see for details \cite{Clarke et al, Giannessi et al}. Let $K$ be a nonempty closed convex subset of a real Hilbert space $\HH$ endowed with an inner product $\langle .,.\rangle$ and its associated norm $\|.\|$, respectively. Suppose $F(.,.):\HH\times \HH\to\RR$ is a nonlinear bifunction, and consider the following problem: find $u\in K$ such that 
\begin{equation}\label{equilibrium}
    F(u,v)\geq 0, \ \ \forall v\in K.
\end{equation}
This is called the equilibrium problem, which was considered and studied by Blum and Oettli \cite{BlumOettli} and modified later by Noor and Oettli \cite{NoorOettli}. The main purpose of this paper is to extend some results of Noor \cite{Noorequilibrium} from considering equilibrium problems to a more general framework pertaining the class of uniformly regular equilibrium problems in which they are defined over a class of nonconvex sets known as uniformly prox-regular sets. We are interested in finding a point $u\in\bar{K}$, where $\bar{K}$ can be nonconvex set and given two positive parameters $k, r >0$, such that 
\begin{equation}\label{unifreg}
    F(u,v)+\frac{k}{2r}\|v-u\|^2\geq 0, \ \ \forall v\in \bar{K}.
\end{equation}
This generalization of equilibrium problem \eqref{equilibrium} is known as the uniformly regular equilibrium model, which was introduced and studied by Poliquin et al. \cite{PoliRockaThiba} and Clarke et al \cite{Clarke et al}.
A very well-known special case of \eqref{equilibrium}, is the so-called classical variational inequality which was first proposed and studied by Stampacchia \cite{Stampacchia}, as a tool for studying partial differential equations in infinite dimensional spaces, when the bifunction $F(u,v)=\langle Tu,v-u\rangle$, where $T:\HH\to \HH$ is a nonlinear operator, i.e., find $u\in K$ such that 
\begin{equation}\label{VI}
\langle Tu,v-u\rangle\geq 0, \ \ \forall v\in K.
\end{equation} 
It will be noted that the proposed algorithms in this work recover the classical variational inequality models.  
\begin{definition}
The proximal normal cone of a closed subset $K\subset\HH$ at $u\in K$, denoted by $N^P(K;u)$, is defined as 
$$N^P(K;u):=\{w\in\HH : \exists \ r>0, \ u\in P_K[u+rw]\}$$
\end{definition}
where the multimapping projection of nearest points in $K$ is defined by
$$P_k[u]=\{u^*\in K: d_K(u)=\|u-u^*\|\},$$
such that $d_K$ is the distance function from $K$ is defined as $d_K(u):=d(u,k)=\inf_{v\in K} \|v-u\|.$
It is not difficult to see that for every  $w\in\HH$, the inclusion $w\in N^P(K;u)$ holds contingent to the existence of a constant $r >0$, and satisfying the following definition:
\medskip 
\begin{definition}
Let $K$ be a nonempty closed subset of $\HH$ and $r\in (0,+\infty ],$ one says that $K$ is uniformly $r$-prox-regular if and only if every nonzero proximal normal to $k$ can be characterized by an $r$-ball, that is, for all $u\in K$ and for all $0 \neq w\in N^P(K;u)\cap \mathbb{B}_r$, we have 
$$\langle w, v-u\rangle\leq \frac{1}{2r}\|v-u\|^2, \ \ \forall v\in K.$$
\end{definition}
Note that if $r=\infty$, then the uniformly prox-regular set $K$ is corresponding with the convex set $K$, hence the uniformly regular nonconvex problem \eqref{unifreg} is reduced to the equilibrium problem \eqref{equilibrium}. We say that $w$ is a proximal normal to $K$ at $u$ with constant $r > 0$. In a general case, and since $K\subseteq\HH$ is a closed set, then $w\in N^P(K;u)$ if and only if there exists a constant $\epsilon >0$ such that 
$$\langle w,v-u\rangle \leq\epsilon\|v-u\|^2, \ \ \forall v\in K.$$

\begin{definition}
The Clarke normal cone, denoted by $N^C(K;u)$, is defined as 
$$N^C(K;u)=\overline{conv}(N^P(K;u)),$$
where $\overline{conv}$ denotes the closure of the convex hull of the proximal normal cone.
\end{definition}
\medskip
If $K$ is convex, it is known that the proximal normal cone $N^P(K;u)$ coincides with the Clarke normal cone in the sense of convex analysis, i.e.,
$N^P(K;u)=\{w\in\HH: \langle w,v-u\rangle\leq 0, \ \forall v\in K\}.$ At the same time, it is clear that if $K$ is uniformly prox-regular, then the proximal normal cone is set-valued mapping (hence closed) and therefore we have $N^P(K;u)=N^C(K;u).$ 
Note that $N^P(K;u)\subset N^C(K;u)$, and $N^C(K;u)$ is always closed and convex while $N^P(K;u)$ is convex but not necessarily closed.
\medskip 
\begin{definition}\label{Fpseudo}
A bifunction $F(.,.):\HH\times \HH\to\RR$ is said to be pseudomonotone on $K\times K$ for the uniformly regular equilibrium problem \eqref{unifreg}, if 
$$F(u,v)+\frac{k}{2r}\|v-u\|^2\geq 0 \implies F(v,u)+\frac{k}{2r}\|v-u\|^2\leq 0, \ \ \ \forall u,v\in K.$$
\end{definition}
\begin{definition}\label{fpseudo}
A mapping $f:\HH\to\HH$ is said to be pseudomonotone on $K$ if 
$$\langle f(u),v-u\rangle+\frac{k}{2r}\|v-u\|^2\geq 0 \implies \langle f(v),v-u\rangle+\frac{k}{2r}\|v-u\|^2\geq 0, \ \ \ \forall u,v\in K.$$
\end{definition}
It is clear from the definitions that the pseudomonotonicity of $f$ in Definition \ref{fpseudo} is equivalent to the pseudomonotonicity of $F$ on $K\times K$ if whenever $F(u,v)=\langle f(u),v-u\rangle.$ Also, it is known that if $f$ is monotone then it is pseudomonotone but the converse not true.  
The concept of merit (gap) functions as a bridge between optimization and equilibrium framework, will be tackled in section 4. A function $g:\HH\to\RR$ is called a regular gap (merit) function if there exists a set $K\subseteq\HH$ such that; 
\begin{itemize}
\item $g$ is nonnegative on $K$, and 
\item $x^*$ solves the equlibrium problem if and only if $x^*\in K$ and $g(x^*)=0$. 
\end{itemize}
\medskip
In general, regular gap functions are neither finite nor smooth no convex. For instance, if the operator $F$ is continuously differentiable on a feasible (bounded polyhedron) set $K$ then $g$ is finite, also if $F$ is monotone and affine (i.e., $F=Ax+b$, where $A$ is positive semi-definite) then $g$ is convex function. To overcome this disadvantage, several regular gap functions were introduced, discussed and analyzed by different authors for abstract equilibrium problems and some classes of variational inequalities, for more details we refer the reader to, for instance \cite{Mastroeni,Fukushima,PappaPassa}, and the references therein. Next, we define a nonlinear map that enables in constructing gap functions and plays an essential role in auxiliary equilibrium problems.
\medskip
\begin{definition}\label{Gbifunction}
Let $G(x,y):\HH\times\HH\to\RR$ be a bifunction satisfies the following properties:  
\begin{equation}\label{Gbifunction}
\begin{aligned}
G(x,y)\geq 0, \ \forall (x,y)\in\HH\times\HH, \\ 
G \ \text{is continuously differentiable,} \\
G(x,.)\ \text{is uniformly strongly convex on} \ K, \\
G(x,x)=G_y'(x,x)=0, \ \ \forall x\in K.
\end{aligned}
\end{equation}
\end{definition}

\section{Iterative Schemes}
In this chapter, we investigate and suggest some iterative schemes for solving uniformly regular equilibrium problems by means of an auxiliary principle technique, which was suggested by Glowinski et al. \cite{Glowinski} and developed by Noor \cite{Noor,Noorequilibrium}. Given a point $u\in \bar{K}$, consider the auxiliary technique of finding a unique $w\in \bar{K}$ such that
\begin{equation}\label{auxiliary}
    \lambda F(w,v)+\langle w-u+\gamma (u-u), v-w\rangle +\frac{k}{2r}\|v-u\|^2\geq 0 , \ \ \forall v\in \bar{K},
\end{equation}
where $\lambda > 0$ and $\gamma > 0$ are constants. Note that if $w=u$, then $w$ can be easily seen as a solution of the equilibrium problem \eqref{equilibrium}. Using the inner product fact $\|x\|^2:=\langle x,x\rangle$ for the last term in inequality \eqref{auxiliary}, along with some suitable modifications, the above auxiliary technique appears as: find $w\in \bar{K}$ such that 
\begin{equation}\label{auxiliarymodified}
\lambda F(w,v)+\big\langle\big (1+\frac{k}{2r}\big )(w-u)+\gamma (u-u)+\frac{k}{2r}(v-w), v-w\big\rangle \geq 0 , \ \ \forall v\in K,
\end{equation}

We observe that this inequality enables us to introduce and analyze the following iterative method for solving \eqref{unifreg}.
\medskip 

\begin{alg}\label{mainalgorithm} 
Let $\{\gamma_n\}$ be nonnegative sequence of real numbers. For a given $u_0\in H$, compute the approximate solution $u_{n+1}$ by the iterative scheme 
\begin{equation}\label{approximate}
\lambda F(u_{n+1},v)+\Big\langle\big (1+\frac{k}{2r}\big )(u_{n+1}-u_n)+\gamma_n (u_n-u_{n-1})+\frac{k}{2r}(v-u_{n+1}), v-u_{n+1}\Big\rangle \geq 0 , \ \ \forall v\in \bar{K},
\end{equation}
\end{alg}
which is known as the inertial proximal method. Such type of inertial proximal algorithms have been used by Alvarez and Attouch \cite{Alvarez} and Noor  for variationa inequalities \eqref{VI}.
\medskip

\begin{remark}\label{proximalalgorithm}
For $\gamma_n=0$, Algorithm \ref{mainalgorithm} reduces to the iterative scheme 
\begin{equation}\label{proximalmethod}
\lambda F(u_{n+1},v)+\Big\langle\big (1+\frac{k}{2r}\big )(u_{n+1}-u_n)+\frac{k}{2r}(v-u_{n+1}), v-u_{n+1}\Big\rangle \geq 0 , \ \ \forall v\in \bar{K},
\end{equation}
which is called the proximal method for solving problem \eqref{unifreg}.
\end{remark}
\medskip

This observation shows that the inertial proximal method includes the classical proximal method as a special case. 
Another auxiliary technique can be noted and used in a problem of finding a unique $w\in \bar{K}$ such that
\begin{equation*}
    \lambda F(u,v)+\langle w-u, v-w\rangle +\frac{k}{2r}\|v-u\|^2\geq 0 , \ \ \forall v\in K.
\end{equation*}
which leads to the following iterative method for given the initial data $u_0\in \bar{K}$, we compute the approximate solution $u_{n+1}$ by the iteration:
\begin{equation*}
\lambda F(u_n,v)+\langle u_{n+1}-u_n, v-u_{n+1}\rangle +\frac{k}{2r}\|v-u_n\|^2\geq 0 , \ \ \forall v\in \bar{K}.
\end{equation*}
We leave the analysis for the latter iterative method to avoid redundancy in our calculations due to the close similarity with proving convergence for algorithm \ref{mainalgorithm} under some minor modifications. 
In the following two theorems we prove the convergence for algorithms \ref{mainalgorithm}, 
and for that, we need the following useful fact related to Euclidean norms:
\medskip

\begin{lemma}\label{lemma}
    Given $u,v\in\HH$. Then $2\langle u,v\rangle =\|u+v\|^2-\|u\|^2-\|v\|^2.$
\end{lemma}
\medskip

\begin{theorem}\label{nonincreasing}
Let $u^*\in \bar{K}$ be a solution of problem \eqref{unifreg} and $u_{n+1}$ be the approximate solution using the proximal method in \eqref{proximalmethod}. For any positive constants $r,k >0$ and $\epsilon=\frac{k}{2r}$, and if $F$ is pseudomonotone, then 
$$\|u_{n+1}-u^*\|^2\leq (1+\epsilon)^2\|u_n-u^*\|^2-\|u_{n+1}-(1+\epsilon)u_n+\epsilon u^*\|^2.$$
\end{theorem}
\begin{proof}
Let $u^*$ be a solution of the uniformly regular equilibrium problem \eqref{unifreg}. Since $F$ is pseudomonotone, then by Definition \ref{Fpseudo}, the following inequality holds
\begin{equation}\label{pseudo1} 
F(v,u^*)+\frac{k}{2r}\|v-u^*\|^2\leq 0, \ \ \forall v\in \bar{K}.
\end{equation}
Take $v=u_{n+1}$ in \eqref{pseudo1} and $v=u^*$ in \eqref{proximalmethod} then we have, respectively, 
\begin{equation}\label{pseudo2} 
F(u_{n+1},u^*)+\frac{k}{2r}\|u_{n+1}-u^*\|^2\leq 0, \ \ \forall v\in \bar{K}.
\end{equation}
and
\begin{equation}\label{proximalmethod2}
\lambda F(u_{n+1},u^*)+\Big\langle\big (1+\frac{k}{2r}\big )(u_{n+1}-u_n)+\frac{k}{2r}(u^*-u_{n+1}), u^*-u_{n+1}\Big\rangle \geq 0 , \ \ \forall v\in \bar{K},
\end{equation}
Combining \eqref{pseudo2} and \eqref{proximalmethod2} together, we have the lower estimate 
\begin{equation}\label{increasinginequality}
\Big\langle\big (1+\frac{k}{2r}\big )(u_{n+1}-u_n)+\frac{k}{2r}(u^*-u_{n+1}), u^*-u_{n+1}\Big\rangle\geq -\lambda F(u_{n+1},u^*)\geq \frac{\lambda k}{2r}\|u_{n+1}-u^*\|^2\geq 0. 
\end{equation}
Use $u=(1+\epsilon)(u_{n+1}-u_n)+\epsilon (u^*-u_{n+1})$ and $v=u^*-u_{n+1}$ in Lemma \ref{lemma}, we have 
\begin{align*}
2\Big\langle\big (1+\frac{k}{2r}\big )(u_{n+1}-u_n)+\frac{k}{2r}(u^*-u_{n+1}) & , u^*-u_{n+1}\Big\rangle =\|(1+\epsilon)(u^*-u_n)\|^2 \\
& -\|(1+\epsilon)(u_{n+1}-u_n)+\epsilon (u^*-u_{n+1})\|^2-\|u^*-u_{n+1}\|^2.
\end{align*}
By combining this last equation with the inequality \eqref{increasinginequality}, the convergence result follows.
\end{proof}

\begin{theorem}
Let $\HH$ be a finite dimensional space. Let $u_{n+1}$ be the approximate solution of algorithm \ref{mainalgorithm} and $u^*$ be the solution of the uniformly regular equilibrium problem in \eqref{unifreg}, then 
$\displaystyle\lim_{n\to\infty}u_n=u^*.$
\end{theorem}
\begin{proof}
Let $u^*\in \bar{K}$ be a solution of \eqref{unifreg} and let $\epsilon >0$ is arbitrarily small. The previous result in Theorem \ref{nonincreasing} showed that $\{\|u^*-u_n\|^2\}$ is nonnegative, $\{\|u^*-u_n\|\}$ is nonincreasing sequence and consequently $\{u_n\}$ is bounded. It also follows that, we have 
$$\sum_{n=0}^\infty\|u_{n+1}-(1+\epsilon)u_n+\epsilon u^*\|^2\leq (1+\epsilon)^2\|u_0-u^*\|^2$$
which, by letting $\epsilon\to 0$, directs to
\begin{equation}\label{convlim}
\lim_{n\to\infty}\|u_{n+1}-u_n\|=0.
\end{equation}
Let $\hat{u}$ be an accumulation point of the successive approximations $\{u_n\}$, hence there exists a subsequence $\{u_{n_k}\}\subseteq\{u_n\}$ that converges to $\hat{u}\in\HH.$ Replace $u_n$ by the subsequence $u_{n_k}$ in \eqref{mainalgorithm} and consider the long-term asymptotic behaviour of the subsequence (i.e., when $n_k\to\infty$) and using \eqref{convlim}, we have
$$ F(\hat{u},v)+\frac{k}{2r}\|v-\hat{u}\|^2\geq 0, \ \forall v\in \bar{K},$$
which implies that $\hat{u}$ solves the uniformly regular equilibrium problem \eqref{unifreg} and
$$\|u_{n+1}-(1+\epsilon)u_n+\epsilon u^*\|^2\leq (1+\epsilon)^2\|u_n-u^*\|^2.$$

Thus it follows from the above inequality that the sequence $\{u_n\}$ has exactly one accumulation point $\hat{u}$ and $\displaystyle\lim_{n\to\infty}u_n=u^*$, the required result.
\end{proof}

\section{Merit functions and descent method for uniformly equilibrium problems}
The descent-type algorithms proposed in this section are a generalization of those proposed by Fukushima for variational inequalities \cite{Fukushima}. 
Merit functions for abstract equilibrium problems are natural generalizations of those known in variational inequalities.   
Hence, we construct the gap (merit) function by utilizing the auxiliary equilibrium which coincides with the original equilibrium problem. This optimization-equilibrium reformulation can be used to develop descent algorithm for solving uniformly regular equilibrium problems under suitable conditions. Recall a regular gap (merit) function $g:\HH\to\RR$ on $\bar{K}\subseteq\HH$ is nonnegative, and $g(u^*)=0$ for any solution $u^*$ of the equilibrium problem if $u^*\in\bar{K}$. Note that $g$ is neither finite nor smooth no convex. It was proved in \cite{Mastroeni} that the equilibrium problem is equivalent to the auxiliary equilibrium problem in which the formulation was termed as; find $u^*\in\bar{K}$ such that 
\begin{equation}\label{AuxiliaryEP}
F(u^*,v)+G(u^*,v)\geq 0, \ \ \forall v\in\bar{K}.    
\end{equation}
A special case for constructing the gap function $G$ is when $G(u,v)=\frac{1}{2}\langle (u-v),M(u-v)\rangle$, was introduced by Fukushima \cite{Fukushima} for variational inequalities, where the matrix $M$ is symmetric and positive definite.
A generalization of the gap function introduced by Fukushima was proposed by replacing the regularizing term $\langle u-v, M(u-v)\rangle /2$ with a general bifunction $G:\HH\times\HH\to\RR$ so that inequality \eqref{AuxiliaryEP} is fulfilled.
For any $\alpha >0,$ the function
$$g_\alpha(u) :=\displaystyle\max_{v\in\bar{K}}\{ F(u,v) -\alpha G(u,v)\}$$
turns out to be a gap function, which is continuously differentiable.
The minimax representation of equilibrium problems in \cite{Mastroeni}, which leads to introducing suitable gap functions acossiated to uinformly regular equilibrium problems, can be characterized by the following lemma.
\medskip

\begin{lemma}
Suppose that $F(v,v)=0, \ \forall v\in K$. Then there exists $u\in K$ such that $F(u,v)\geq 0, \ \forall v\in K$ if and only if 
$$\min_{u}\max_{v}\{-F(u,v)\}=0.$$
\end{lemma}
Next, define the functional $\phi (w):=\frac{1}{2}\langle w-u,w-u \rangle -\lambda F(u,w),$ over the convex set $K$, this is known as the auxiliary energy/cost potential function associated with the auxiliary equilibrium problem.
In the light of the previous formulation we are guided to conclude that uniformly regular equilibrium problems can be reformulated in terms of multiobjective optimization problems as
$$\min_{u\in K}\max_{v\in K}\{-F(u,v)-\frac{k}{2r}\|v-u\|^2\}$$ 
provided that tha optimal value is zero; this leads to consider the following regular gap function for uniformly regular equilibrium problem considered in \eqref{unifreg}:
$$g(u)=\sup_{v}\{F(u,v)-\frac{k}{2r}\|v-u\|^2\}.$$

Hence, it is obvious that the uniformly regular equilibrium problem is a special case of the auxiliary equilibrium problem \eqref{AuxiliaryEP} subject to $G=(-k/2r) \|u-w\|^2$,\ which leads to finding the minimum of the energy functional $\phi$ and reformulating the uniformly regular equilibrium problem \eqref{unifreg} as an optimization problem: 
\begin{equation}\label{gapforUREP}
\max_{w\in K}\{F(u,w)-(\alpha/2)\|u-w\|^2\},
\end{equation}
where $\alpha=\frac{k}{r} >0$ is a constant.  
One of the obstacles that we encounter when dealing with gap functions is that differentiability is not guaranteed. To this end, Fukushima [7] defined a continuously differentiable gap function for variational inequalities and later was generalized to equilibrium problems. By using the equivalence between the original equilibrium and its associated auxiliary equilibrium in \cite{Mastroeni}, this allows to regularize the original problem. Hence, we state sufficient conditions that guarantee the differentiability of the gap function associated to the auxiliary equilibrium which can be asserted by the following theorem.
\medskip
\begin{theorem}\label{gaptheorem}
Let $\bar{K}$ be a closed subset of $\HH$, and $F(x,y)$ be a lower semi-continuous strictly convex smooth function with respect to $y$ for all $x\in \bar{K}$, and smooth with respect to $x$ such that $f'_x$ is continuous on $\bar{K}\times\bar{K}.$ Assume that the assumptions in \eqref{Gbifunction} for the bifunction $G$ are satisfied. Assume that the optimization problem is defined and the maximum is attained. Then 
\begin{equation*}\label{gap}
g(x):=\max_{y}\{-F(x,y)-G(x,y)\}
\end{equation*}
is continuously differentiable gap function for the auxiliary equilibrium problem and its gradients given by 
$$g'(x):=-F'_x(x,y(x))-G'_x(x,y(x))$$ where 
$y(x):=\displaystyle\argmin_{y}\{F(x,y)+G(x,y)\}$
\end{theorem}
\begin{proof}
In order to obtain a continuously differentiable gap function for auxiliary equilibrium problems we use a proposition in \cite{Mastroeni} such that it proves the equivalence between equilibrium problems and their auxiliary ones, hence the existence result for gap functions follows.
\end{proof}
Apply this theorem, in our case, to uniformly regular equilibrium problem then it is clear that \eqref{gapforUREP} is actually the associated gap function, i.e.,
\begin{equation}\label{gapmeritUREP}
g(u):=\max_{w\in K}\{F(u,w)-(\alpha/2)\|u-w\|^2\}
\end{equation}
and its gradients 
\begin{equation*}
g'(u):=F'_u(u,w)-\frac{d}{du}\Big(\frac{\alpha}{2}\|u-w\|^2\Big)
\end{equation*}
\begin{algorithm}
\begin{alg}[Descent Method]\label{descentalgorithm}
Let $u_0\in\bar{K}$, and let the gap function $g$ be defined as in \eqref{gapmeritUREP} \\
\medskip

\textbf{Step.1} \ (iteration) \hspace{3cm} $u_{n+1}:=u_n+t_nd_n$ \\ 

where $d_n=w(u_n)-u_n,$ is the descent direction for $g$ at $u_n$, provided that $$w(u_n):=\displaystyle\argmin_{w}\{F(u_n,w)+\frac{\alpha}{2}\|w-u_n\|^2\},$$ and $t_n$ is the solution of $$\displaystyle\min_{t\in [0,1]}g(u_n+td_n).$$
\\ 
\medskip 
\textbf{Step.2} \ If $\|u_{n+1}-u_n\|<\alpha$ for some fixed $\alpha >0$ then \textbf{STOP},

otherwise put $n=n+1$ and repeat \textbf{Step.1}
\end{alg}
\end{algorithm}

To show that $d_n$ is a descent direction for $g$ at the point $x_n.$
We assume the following necessary condition: 
\begin{equation}\label{necessarycondition}
\langle F'_u(u,w)+\frac{d}{du}\Big(\frac{\alpha}{2}\|u-w\|^2\Big)+ F'_w(u,w)+\frac{d}{dw}\Big(\frac{\alpha}{2}\|u-w\|^2\Big),u-w\rangle\geq 0, \ \ \forall(u,w)\in\bar{K}\times\bar{K}.
\end{equation}
Note that if $F(u,w):=\langle T(u),w-u\rangle$ then the above necessary condition holds provided that $\ \nabla T(u)$ is a positive definite matrix, for all $u\in\bar{K}$, and it also holds if $F$ is strictly convex on $\bar{K}$ with respect to $w$.
\medskip 
\begin{proposition}
Suppose that the hypotheses of Theorem \ref{gaptheorem} hold and moreover the assumption \eqref{necessarycondition} is applied. Then the distance function $d(u):=w(u)-u$ is a descent direction for $g$ at $u\in\bar{K}$, provided that $w(u)\not=u.$
\end{proposition}
\begin{proof}
It is clear to observe that $u^*$
is a solution for equilibrium problem if and only if $w(u^*)=u^*$. The strictly convexity of $F$ with respect to $w$ implies that there exists a unique minimum point $w(u)=\arg\displaystyle\min_{w}F(u,w)$, which satisfies the following variational inequality holds:
$$\langle F'_w(u,w(u))+\frac{d}{dw}\Big(\frac{\alpha}{2}\|u-w(u)\|^2\Big),z-w(u)\rangle\geq 0, \ \ \forall z\in\bar{K}.$$
Setting $z:=u$, we have $\langle F'_w(u,w(u))+\frac{d}{dw}\Big(\frac{\alpha}{2}\|u-w(u)\|^2\Big),w(u)-u\rangle < 0.$ Combining the latter inequality with \eqref{necessarycondition}, we obtain 
$$-\langle F'_u(u,w(u))+\frac{d}{du}\Big(\frac{\alpha}{2}\|u-w\|^2\Big),u-w(u)\rangle\leq\langle F'_w(u,w(u))+\frac{d}{dw}\Big(\frac{\alpha}{2}\|u-w(u)\|^2\Big),u-w(u)\rangle <0.$$
Invoking Theorem \ref{gaptheorem} and using the gradient 
\begin{equation*}
g'(u):=F'_u(u,w)-\frac{d}{du}\Big(\frac{\alpha}{2}\|u-w\|^2\Big),
\end{equation*}
we obtain $\langle g'(u),u-w(u)\rangle <0.$

\end{proof}
\begin{theorem}\label{convergence thm}
Let $G$ be defined as in Definition \ref{Gbifunction}. Suppose that $\bar{K}$ is a compact set in $\HH$ and assume that the necessary condition \eqref{necessarycondition} is valid.
Then for any given initial point $u_0\in\bar{K}$, the sequence $\{u_n\}$ generated by Algorithm \ref{descentalgorithm} belongs to $\bar{K}$ and any accumulation point of $\{u_n\}$ is a solution to the uniformly regular equilibrium problem. 
\end{theorem}
\begin{proof}
The convexity of $\bar{K}$ indicates that the generated sequence $\{u_n\}$ belongs to $\bar{K}$ for all values $0\leq t_n\leq 1$. Since the uniformly regular $F(u,w)+\frac{\alpha}{2}\|w-u\|^2$ is strictly convex with respect to $w$ then there exists a unique minimum point $$w(u):=\arg\displaystyle\min_{w} \{F(u,w)+\frac{\alpha}{2}\|w-u\|^2\}.$$ Applying
Theorem $4.3.3$ of \cite{Bank}, we conclude that $w(u)$ is single-valued and continuous at $u$, which implies that the function $d(u) := w(u)-u$ is continuous on $\bar{K}$. It is known that the set 
$$\Phi(u;d)=\{w:w=u+t_nd,\ g(w)=\min_{0\leq t\leq 1}g(u+td)$$ is closed if $g$ is a continuous map. Hence, one can suggest the iteration $u_{n+1}=\Phi(u_n;d(u_n))$ and show that is itself closed, see for instance \cite{Minoux}. By  Zangwill’s convergence theorem \cite{Zangwill}, it is obvious that any accumulation point of the sequence $\{u_n\}$ is a solution of the uniformly regular equilibrium problem \eqref{unifreg}. 
\end{proof}
\medskip

\section{conclusion}
The auxiliary principle technique, used in \cite{Glowinski, Noor} for variational inequalities and equilibrium problems, has been adopted to develop some new implicit iterative algorithms (a.k.a proximal-type methods) along with their convergence analysis for solving uniformly regular equilibrium problems that are defined on prox-regular nonconvex sets. The associated operator is pseudomonotone, and therefore extending the results of this work to the case of quasimonotonicity is promising. Every pseudomonotone is quasimonotone but the converse not true. 

Moreover, gap (merit) functions were constructed for uniformly regular equilibrium problems based on the role it plays for bridging equilibrium and optimization problems to investigate and suggest some descent-type methods for the corresponding formulation. The convergence of descent methods developed for abstract equilibrium problems is usually based on smoothness assumptions. We think that some efforts should be devoted to develop algorithms for nonsmooth problems, which include nonsmooth Nash equilibrium problems as special cases. The proposed descent algorithm is a natural extension of those improved for variational inequalities \cite{Fukushima}. The analysis of the gap function approach for equilibrium problems allows to extend the applications to further variational formulations as, for instance, the Minty (dual) Variational Inequality which can be formulated as finding 
$$x^*\in K \ s.t. \ \langle F(y),y-x^*\rangle\geq 0, \ \ \forall y\in K$$
which, under the continuity and pseudomonotonicity of the operator F, is equivalent to variational inequalities. Moreover, it would be interesting to extend the Moreau-Yosida regularization to merit functions for abstract equilibrium problems. Another future research direction is pertaining the mixed variational equilibrium problem which covers almost all other problems and treat them as special cases. Finally, we believe that new merit functions could be developed without assuming the convexity; this might allow to extend the merit function approach to nonconvex Nash equilibrium problems and the generalized Nash equilibrium problems (equivalently, quasi-variational inequalities where the feasible set is convex-valued constraint map).


\medskip

\subsection*{Declarations} 
The author declares that there was no conflict of interest or competing interest.



\end{document}